\documentclass[preprint]{svjour3}

\let\vec\relax
\DeclareMathAccent{\vec}{\mathord}{letters}{"7E}

\usepackage{amsmath}

\usepackage{amsfonts}
\usepackage{amssymb}
\usepackage{mathtools}
\usepackage{yfonts}
\usepackage{xcolor}

\usepackage{graphicx}
\usepackage{epstopdf}

\usepackage[ruled,vlined,linesnumbered]{algorithm2e}

\usepackage{accents}
\newcommand*{\dt}[1]{ \accentset{\mbox{\large\bfseries .}}{#1}}
\DeclareMathOperator{\bigtimes}{{\hbox{\large\sf X}}}

\def\R{{\mathbb R}}
\def\C{{\mathbb C}}
\def\eps{\varepsilon}

\def\P{\mathrm{P}}
\newcommand\calM{\mathcal{M}}

\newcommand\bfr{{\mathbf r}}

\newcommand\bfI{{\mathbf I}}
\newcommand\bfA{{\mathbf A}}
\newcommand\bfB{{\mathbf B}}

\newcommand\bfD{{\mathbf D}}
\newcommand\bfF{{\mathbf F}}
\newcommand\bfS{{\mathbf S}}
\newcommand\bfU{{\mathbf U}}
\newcommand\bfV{{\mathbf V}}
\newcommand\bfY{{\mathbf Y}}
\newcommand\bfZ{{\mathbf Z}}

\def\eps{\varepsilon}

\def\phi{\varphi}

\author{Gianluca Ceruti, Christian Lubich}
\title{Time integration of symmetric and anti-symmetric \\
low-rank matrices and  Tucker tensors}
\date{}

\institute{G. Ceruti and Ch. Lubich \at
	Mathematisches Institut, Universit{\"a}t T{\"u}bingen, Auf der Morgenstelle 10, D-72076 T{\"u}bingen, Germany.
	\email{\{ceruti, lubich\}@na.uni-tuebingen.de}           
}

\begin{document}

\maketitle

\begin{abstract} A numerical integrator is presented that computes a symmetric or skew-symmetric low-rank approximation to large symmetric or skew-symmetric time-depen\-dent matrices that are either given explicitly or are the unknown solution to a matrix differential equation. A related algorithm is given for the approximation of symmetric or anti-symmetric time-dependent tensors by symmetric or anti-symmetric Tucker tensors of low multilinear rank. The proposed symmetric or anti-symmetric low-rank integrator is different from recently proposed projector-splitting integrators for dynamical low-rank approximation, which do not preserve symmetry or anti-symmetry. However, it is shown that the (anti-)symmetric low-rank integrators retain favourable properties of the projector-splitting integrators: low-rank time-dependent matrices and tensors are reproduced exactly, and the error behaviour is robust to the presence of small singular values, in contrast to standard integration methods applied to the differential equations of dynamical low-rank approximation.
Numerical experiments illustrate the behaviour of the proposed integrators.
\end{abstract}

\section{Introduction}
In this paper we propose and analyse an algorithm that computes a symmetric or skew-symmetric low-rank approximation to large symmetric or skew-symmetric time-depen\-dent matrices that are either given explicitly or are the unknown solution to a matrix differential equation. A related algorithm is given for the approximation of symmetric or anti-symmetric time-dependent tensors by symmetric or anti-symmetric Tucker tensors of low multilinear rank.

In the matrix case,  motivation for this work comes from Lyapunov and Riccati differential equations, which have large symmetric matrices as solutions, which can often be well approximated by low-rank matrices \cite{mena2018numerical}. For tensors, our main motivation comes from the quantum dynamics of bosonic or fermionic systems, where the symmetric or anti-symmetric wave function is  approximated by low-rank symmetric or anti-symmetric Tucker tensors in the
MCTDHB and MCTDHF methods for bosons and fermions, respectively \cite{AlonSC08,Caillat:MCTDHF}. An efficient integrator that preserves symmetry and anti-symmetry and uses them to reduce the computational complexity, is needed in these and other applications, such as using a step of the integrator as a computationally efficient retraction in optimization algorithms for (anti-)symmetric low-rank matrices and tensors.

The algorithms proposed in this paper are non-trivial modifications of the pro\-jector-splitting integrators for the dynamical low-rank approximation of matrices and Tucker tensors that were proposed in \cite{LubichOseledets} and \cite{Lubich:MCTDH,LubichVandWalach}, respectively. The projector-splitting integrators have been shown to possess remarkable robustness to the typical presence of small singular values \cite{KieriLubichWalach,LubichVandWalach}, as opposed to applying standard integrators  to the differential equations of dynamical low-rank approximation that are given in \cite{KochLubich07,KochLubich10}. However, the projector-splitting integrators do not preserve symmetry or anti-symmetry.

We will show that the (anti-)symmetry-preserving integrators proposed here retain the robustness with respect to small singular values of the projector-splitting algorithms. This relies on an exactness property, namely that time-dependent matrices and tensors of the approximation rank are reproduced exactly by the integrator. This exactness property will also be shown to be retained from the projector-splitting integrators. We note, however, that the integrators proposed here can no longer be interpreted as splitting integrators.

The new (anti-)symmetry-preserving integrators are favourable also from the computational viewpoint: compared with the projector-splitting integrator, the computational cost is halved in the (skew-)symmetric matrix case; in the case of $d$-dimensional (anti-)symmetric tensors, the computational cost for the core tensor is reduced by the factor $1/d!$, and that for the basis matrices by $1/d$.

A first attempt to modify the projector-splitting integrator and preserve the symmetry in the matrix setting, can be found in \cite{mena2018numerical}: numerical examples show the correct behaviour of the approximate solution but no convergence analysis or extension to multi-dimensional arrays is provided, and no use of the symmetry is made to reduce the computational effort.

The outline of the paper is the following: in Section 2, we briefly restate the idea of dynamical low-rank approximation for matrices and we present the matrix projector-splitting integrator with some of its properties. In Section 3, we consider the case of (skew-)symmetric matrices; we present the (skew-)symmetry-preserving low-rank integrator and study its properties. In Section 4, we recapitulate the projector-splitting integrator for low-rank Tucker tensors. In Section 5, we present the integrator for (anti)-symmetric tensors of low multilinear rank and study its properties. In the final section, we present numerical experiments that illustrate the approximation properties and the robustness to small singular values.

Throughout the paper, we use the convention to denote matrices by boldface capital letters and tensors by italic capital letters.
%

\section{General matrices: recap of the projector-splitting integrator for dynamical low-rank approximation}
The objective is to approximate large time-depend\-ent matrices $\bfA(t)\in \R^{m\times n}$  for $0\le t \le T$ by rank-$r$ matrices $\bfY(t)$ with comparatively low rank $r\ll m,n$, which require much less storage than $\bfA(t)$ when they are available in a factorized, SVD-like form. The large, or often too large matrices $\bfA(t)$ may be given explicitly or they are the unknown solution to a matrix differential equation
(with right-hand side function $\bfF:\R\times \R^{m\times n} \to \R^{m\times n}$)
\begin{equation} \label{eq:fullEq-mat}
\dt{\bfA}(t) = \bfF(t, \bfA(t)), 
\qquad
\bfA(t_0) = \bfA_0 .
\end{equation}
Dynamical low-rank approximation as presented in  \cite{KochLubich07} determines $\bfY(t)$ as the solution of a projected matrix
differential equation,
with a projection $\P(\bfY)$ onto the tangent space $T_\bfY \calM_r$ of the manifold of rank-$r$ matrices at $\bfY\in\calM_r$,
\begin{equation} \label{eq:projEq}
\dt{\bfY}(t) = \P(\bfY(t))\bfF(t, \bfY(t)),
\qquad
\bfY(t_0) = \bfY_0,
\end{equation}
where $\bfY_0$ is a rank-$r$ approximation to $\bfA_0$, typically obtained by a truncated singular value decomposition.
(Here, $\bfF(t,\bfY) = \dt\bfA(t)$ if $\bfA(t)$ is given explicitly.) The solution $\bfY(t)$ to this projected matrix differential equation then stays in the rank-$r$ manifold~$\calM_r$.

To make this abstract formulation practically useful,  rank-$r$ matrices $\bfY(t)$ are written (non-uniquely) in factored form
\begin{equation} \label{USV}
\bfY(t) = \bfU(t)\bfS (t)\bfV(t)^\top, 
\end{equation}
where the slim matrices $\bfU(t)\in \R^{m\times r}$ and $\bfV(t)\in \R^{n\times r}$ each have $r$ orthonormal columns, 
and the small square matrix $\bfS(t)\in \R^{r\times r}$ is invertible. We choose the tangent space projection $\P(\bfY)$ as the orthogonal projection onto $T_\bfY(\calM_r)$ with respect to the Euclidean or Frobenius inner product $\langle \bfA,\bfB \rangle = \textbf{vec}(\bfA)^\top\textbf{vec}(\bfB)$, where $\textbf{vec}(\bfA)\in\R^{mn}$ is a vectorization of $\bfA$.
Then, $\P(\bfY)$ is given as an alternating sum of three subprojections \cite{KochLubich07},
\begin{equation}\label{P}
\P(\bfY)\textbf{Z} = \textbf{ZVV}^\top - \textbf{UU}^\top \textbf{Z VV}^\top + \textbf{UU}^\top\textbf{Z} .
\end{equation}
The projector-splitting integrator of \cite{LubichOseledets} is a Lie--Trotter or Strang splitting method that splits the right-hand side of \eqref{eq:projEq} according to the three terms in \eqref{P}. It turned out that such a splitting combines very well with the factorization~\eqref{USV}. In the first substep of a Lie--Trotter splitting, $\textbf{K}:=\bfU\bfS$ is updated, in the second substep $\bfS$ is updated, and in the third substep $\textbf{L}:=\bfV\bfS^\top$. The algorithm alternates between the numerical solution of matrix differential equations (of dimensions $m\times r$, $r\times r$, $n\times r$) and orthogonal decompositions of slim matrices (of dimensions $m\times r$ and $n\times r$).
One time step of integration from time $t_0$ to $t_1=t_0+h$  starting from a factored rank-$r$ matrix 
$\bfY_0=\bfU_0\bfS_0\bfV_0^\top$ proceeds as follows:

\begin{enumerate}
	\item 
	\textbf{K-step} : Update $ \bfU_0 \rightarrow \bfU_1, \bfS_0 \rightarrow \hat{\bfS}_1$ \\
	Integrate from $t=t_0$ to $t_1$ the $m \times r$ matrix differential equation
	$$ \dot{\textbf{K}}(t) = \bfF(t, \textbf{K}(t) \bfV_0^\top) \bfV_0, \qquad \textbf{K}(t_0) = \bfU_0 \bfS_0.$$
	Perform a QR factorization $\textbf{K}(t_1) = \bfU_1 \hat{\bfS}_1$. 
		
	\item
	\textbf{S-step} : Update $ \hat{\bfS}_1 \rightarrow \tilde{\bfS}_0$ \\
	Integrate from $t=t_0$ to $t_1$ the $r \times r$ matrix differential equation
	$$ \dot{\bfS}(t) = - \bfU_1^\top \bfF(t, \bfU_1 \bfS(t) \bfV_0^\top) \bfV_0, \qquad \bfS(t_0) = \hat{\bfS}_1,$$
	and set $\tilde{\bfS}_0 =\bfS(t_1)$.
	
	\item
	\textbf{L-step} : Update $ \bfV_0 \rightarrow \bfV_1, \tilde{\bfS}_0 \rightarrow \bfS_1$ \\
	Integrate from $t=t_0$ to $t_1$ the $n \times r$ matrix differential equation
	$$  \dot{\textbf{L}}(t) =\bfF(t, \bfU_1 \textbf{L}(t)^\top)^\top  \bfU_1, \qquad \textbf{L}(t_0) = \bfV_0 \tilde{\bfS}_0^\top. $$
	Perform a QR factorization $\textbf{L}(t_1) = \bfV_1 \bfS_1^\top$.
\end{enumerate} 

Then, the approximation after one time step is given by 
$$ \bfY_1 = \bfU_1 \bfS_1 \bfV_1^\top .$$
To proceed, we iterate the procedure taking $\bfY_1$ as starting point for the next step. 

The above algorithm describes the first-order Lie--Trotter splitting. The algorithm for the second-order Strang splitting is obtained by concatenating the above algorithm with the same algorithm in reversed order, each for half the step-size; see \cite{LubichOseledets} for the detailed description.

The projector-splitting integrator has remarkable properties. First, it reproduces rank-$r$ matrices without error.

\begin{theorem}[{Exactness property, \cite[Theorem 4.1]{LubichOseledets}}]
\label{thm:proj-split-exact}
	Let $\bfA(t) \in \mathbb{R}^{m \times n}$ be of rank~$r$  for $t_0 \leq t \leq t_1$,
	so that $\bfA(t)$ has a factorization \eqref{USV}, $\bfA(t)=\bfU(t)\bfS(t)\bfV(t)^\top$. 
	Moreover, assume  that the $r\times r$ matrix $ \bfV(t_1)^\top \bfV(t_0)$ is invertible. With $\bfY_0 = \bfA(t_0)$, the projector-splitting integrator for 
	$\dt\bfY(t)=\P(\bfY(t))\dt \bfA(t)$ is then exact: $ \bfY_1 = \bfA(t_1)$.
\end{theorem}

The second remarkable property is the robustness of the algorithm to the presence of small singular values of the solution or its approximation. This is in contrast to standard integrators applied to \eqref{eq:projEq} or the equivalent differential equations for the factors $\bfU(t)$, $\bfS(t)$, $\bfV(t)$, which contain a factor $\bfS(t)^{-1}$ on the right-hand sides \cite[Prop.\,2.1]{KochLubich07}. Moreover, the local Lipschitz constant of the tangent space projection $\P(\cdot)$ is proportional to the inverse of the smallest nonzero singular value \cite[Lemma 4.2]{KochLubich07}.
The appearance of small singular values is typical in applications, because the smallest singular value retained in the approximation cannot be expected to be much larger than the largest discarded singular value of the solution, which needs to be small to obtain good accuracy of the low-rank approximation.

\begin{theorem}[{Robust error bound, \cite[Theorem 2.1]{KieriLubichWalach}}]
\label{thm:proj-split-robust}
	Let $\bfA(t)$ denote the solution of the matrix differential equation \eqref{eq:fullEq-mat}. Assume that  the following conditions hold in the Frobenius norm $\|\cdot\|=\|\cdot\|_F$:
	\begin{enumerate}
		\item 
		$\bfF$ is Lipschitz-continuous and bounded: for all $\bfY, \widetilde{\bfY} \in \mathbb{R}^{m \times n}$ and $0\le t \le T$,
		$$ 
			\| \bfF(t, \bfY) - \bfF(t, \widetilde{\bfY}) \| 
			\leq
			L \| \bfY - \widetilde{\bfY} \|,
			\qquad
			\| \bfF(t, \bfY) \| \leq B \ .
		$$
		
		\item
		The non-tangential part of $\bfF(t, \bfY)$ is $\varepsilon$-small:
		$$
		\| (\bfI - \P(\bfY)) \bfF(t, \bfY) \| \le \eps
		$$
		for all $\bfY \in \mathcal{M}$ in a neighbourhood of $\bfA(t)$ and $0\le t \le T$.
		
		\item
		The error in the initial value is $\delta$-small:
		$$
		\| \bfY_0 - \bfA_0 \| \le \delta.
		$$
	\end{enumerate}	
	Let $\bfY_n$ denote the rank-$r$ approximation to $\bfA(t_n)$ at $t_n=nh$ obtained after n steps of the projector-splitting integrator with step-size $h>0$.
	Then, the error satisfies for all $n$ with $t_n =  nh \leq T$
	$$ \| \bfY_n - \bfA(t_n) \| \leq c_0\delta + c_1 \varepsilon + c_2 h ,$$	
	where the constants $c_i$ only depend on $L, B,$ and $T$. In particular, the constants are independent of singular values of the exact or approximate solution. 
\end{theorem}

It is further shown in \cite[Section 2.6.3]{KieriLubichWalach} that an inexact solution of the matrix differential equations in the projector-splitting integrator
leads to an additional error that is bounded in terms of the local errors in the inexact substeps, again with constants that do not depend on small singular values.

Numerical experiments with the matrix projector-splitting integrator and comparisons with standard numerical integrators are reported in \cite{LubichOseledets,KieriLubichWalach}. These experiments show good behaviour also for spatially discretized partial differential equations where the Lipschitz constant becomes large, a case that as of now is not covered by the theory.

\section{Symmetric and skew-symmetric matrices: a structure-preserving integrator for dynamical low-rank approximation} 
\label{sec:Sym}
We now assume that the right-hand side function in \eqref{eq:fullEq-mat} is such that
\begin{equation} \label{F-sym}
\text{$\bfF(t,\bfY)$ is (skew-)symmetric whenever $\bfY$ is (skew-)symmetric.}
\end{equation}
This condition ensures that the solutions to the matrix differential equation \eqref{eq:fullEq-mat} and the projected differential equation \eqref{eq:projEq} are (skew-)symmetric provided the initial values are (skew-)symmetric. For  \eqref{eq:projEq}, this is seen from formula \eqref{P} for the tangent space projection with equal left and right factors $\bfV=\bfU$ in the decomposition \eqref{USV} of the (skew-)symmetric rank-$r$ matrix $\bfY=\bfU\bfS\bfU^\top$.

While the projector-splitting integrator for dynamical low-rank approximation described in the previous section has favourable properties, it does \emph{not} preserve symmetry or skew-symmetry of the solution $\bfA(t)$ to \eqref{eq:fullEq-mat}.

\subsection{(Skew-)symmetry preserving integrator} 
We now propose a modified integrator that preserves symmetry and skew-symm\-etry and still retains the exactness and robustness properties of the projector-splitting integrator.
A step with this integrator consists of two substeps. The first substep is identical to the first substep ($\textbf{K}$-step) of the projector-splitting integrator: it updates $\textbf{K}=\bfU\bfS$ in the decomposition $\bfY=\bfU \bfS \bfU^\top$. The second substep is a substantially modified update of $\bfS$, which can be viewed as a Galerkin approximation in the basis provided by the first substep.

Given $\bfY_0 = \bfU_0 \bfS_0 \bfU_0^\top$ with a (skew-)symmetric $r\times r$-matrix $\bfS_0$ at time $t_0$, we compute  the factorization $\bfY_1 = \bfU_1 \bfS_1 \bfU_1^\top$ with a (skew-)symmetric $r\times r$-matrix $\bfS_1$ at time $t_1=t_0+h$ by the following algorithm:

\hfill \break
\begin{algorithm}[H]\label{alg:nestedKSL}
	\caption{One time step of the (skew-)symmetry preserving integrator}
	\label{algntucker}
	\KwData{ $\bfY^0 = \bfU_0 \bfS_0 \bfU_0^\top $ in factorized form, function $\bfF(t,\bfY)$, $t_0$, $t_1$}
	\KwResult{ $\bfY_1 = \bfU_1 \bfS_1 \bfU_1^\top$ in factorized form}
	\Begin{
	 Integrate from $t=t_0$ to $t_1$ the $n \times r$ matrix differential equation
	 		\begin{equation*}
	 			\dot{\textbf{K}}(t) = \bfF(t, \textbf{K}(t) \bfU_0^\top)  {\bfU}_0,
	 			\qquad
	 			\textbf{K}(t_0) = \bfU_0 \bfS_0.
	 		\end{equation*}
	 		
	 		Compute a QR-factorization $ \textbf{K}(t_1) = \bfU_1 \textbf{R}$.
	 		
	 		\
	 
	 		Integrate from $t=t_0$ to $t_1$ the $r \times r$ differential equation
	 				\begin{align*}
	 					&\dot{\bfS}(t) = \bfU_1^\top \bfF(t, \bfU_1 \bfS(t) \bfU_1^\top)  {\bfU}_1, 
	 					\\
	 					&\bfS(t_0) = \bfU_1^\top \bfY_0  {\bfU}_1 = ( \bfU_1^\top\bfU_0) \bfS_0 ( \bfU_1^\top\bfU_0)^\top.
	 							\end{align*}
	 		Set $\bfS_1=\bfS(t_1)$.
	} 
\end{algorithm}
\vspace{10pt}

%
	To continue in time,  we take $\bfY_1$ as starting value for the next step and perform another step of the integrator. 
	
	Note that in this integrator the factor $\textbf{R}$ in the $QR$-decomposition of the first substep is not reused in the second substep, in contrast to the projector-splitting integrator. The computational cost is approximately halved, since the $\textbf{L}$-step is not needed here.
		
	We will now show that the (skew-)symmetric integrator retains the exactness and robustness properties of the projector-splitting integrator, using these known results in the proof.

\subsection{Exactness property of the (skew-)symmetric integrator} The exactness result
Theorem~\ref{thm:proj-split-exact} extends in the following way.

\begin{theorem}[Exactness property]
\label{thm:sym-exact}
	Let $\bfA(t) \in \mathbb{R}^{n \times n}$ be (skew-)symmetric and of rank~$r$  for $t_0 \leq t \leq t_1$,
	so that $\bfA(t)$ has a factorization \eqref{USV} with equal left and right factors, $\bfA(t)=\bfU(t)\bfS(t)\bfU(t)^\top$. 
	Moreover, assume  that the $r\times r$ matrix $ \bfU(t_1)^\top \bfU(t_0)$ is invertible. With $\bfY_0 = \bfA(t_0)$, the (skew-)symmetric integrator for 
	$\dt\bfY(t)=\P(\bfY(t))\dt \bfA(t)$ is then exact: $ \bfY_1 = \bfA(t_1)$.
\end{theorem}

\begin{proof}
	We note that the projector-splitting integrator and the (skew-)symmetric integrator have the same first step. Let $\bfU_1 \in \mathbb{R}^{n \times r}$ be the matrix with orthonormal columns computed in the first substep. Due to the exactness of the matrix projector-splitting integrator as given by Theorem~\ref{thm:proj-split-exact} we know  that $\bfU_1$ and $\bfA(t_1)$ have the same range and therefore
\begin{equation} \label{UUA}
\bfU_1 \bfU_1^\top \bfA(t_1) = \bfA(t_1). 
\end{equation}
	Denoting $\Delta \bfA = \bfA(t_1) - \bfA(t_0)$, the (skew-)symmetric integrator provides for the second substep the solution
	$$ \bfS_1 = \bfU_1^\top \bfY_0 {\bfU}_1 + \bfU_1^\top (\bfA(t_1) - \bfA(t_0)) {\bfU}_1 =
	\bfU_1^\top \bfA(t_1)   {\bfU}_1, $$
	since $\bfY_0=\bfA(t_0)$.
	The result  after a time step  of the (skew-)symmetric integrator is
	$$ \bfY_1 =  \bfU_1 \bfS_1 \bfU_1^\top=\bfU_1 \bfU_1^\top \bfA(t_1)  (\bfU_1 \bfU_1^\top) = \bfA(t_1) ,$$
	where the last equality holds because of \eqref{UUA} and the (skew-)symmetry of $\bfA(t_1)$. \qed
\end{proof}

\subsection{Robustness to small singular values}
The error bound of
Theorem~\ref{thm:proj-split-robust} extends in the following way.

\begin{theorem}[{Robust error bound}]
\label{thm:sym-robust}
	Let $\bfA(t)$ denote the (skew-)symmetric solution of the matrix differential equation \eqref{eq:fullEq-mat} with $\bfF$ satisfying \eqref{F-sym}. Assume  that conditions $1$.-$\,3$. of Theorem~$\ref{thm:proj-split-robust}$ are fulfilled.
	
	Let $\bfY_n$ denote the rank-$r$ approximation to $\bfA(t_n)$ at $t_n=nh$ obtained after n steps of the (skew-)symmetric integrator of Algorithm~\ref{alg:nestedKSL} with step-size $h>0$.
	Then, the error satisfies for all $n$ with $t_n =  nh \leq T$
	$$ \| \bfY_n - \bfA(t_n) \| \leq c_0\delta + c_1 \varepsilon + c_2 h ,$$	
	where the constants $c_i$ only depend on $L, B,$ and $T$. In particular, the constants are independent of singular values of the exact or approximate solution. 
\end{theorem}

As in \cite[Section 2.6.3]{KieriLubichWalach}, it can be further shown that an inexact solution of the matrix differential equations in the projector-splitting integrator
leads to an additional error that is bounded in terms of the local errors in the inexact substeps, again with constants that do not depend on small singular values.

\begin{remark}
The method of Algorithm~\ref{alg:nestedKSL} is of order~1, and higher order can be obtained simply by composition as, e.g., in 
\cite[Section II.4]{HLW}. 
However, like for the projector-splitting integrator of \cite{LubichOseledets}, it is not known if an error bound of  higher order in the step-size $h$ can be obtained with constants that are independent of small singular values. Numerical experiments with the Strang version of the projector-splitting integrator, which is of order~2,  indicate an order reduction in some examples with very small singular values \cite{ostermann2018convergence}.
\end{remark}
%
%
%
%
%
%
%
%
%
We now prepare for the proof of Theorem~\ref{thm:sym-robust}, which views the (skew-)symmetric integrator as a perturbed variant of the projector-splitting integrator.

Let us introduce the quantity   
$$  \vartheta(h, \varepsilon) := (4e^{Lh} BL  + 9BL)h^2 + (3e^{Lh}+4)\varepsilon h \ , $$
which represents the local error bound after one time step of the projector-splitting integrator, as proved in  \cite[Theorem 2.1]{KieriLubichWalach}.

In the following, we denote by $\bfU_1\in\R^{n\times r}$ the matrix with orthonormal columns obtained in the first substep of the integrator. We recall that the matrix projector-splitting and the (skew-)symmetric integrator have the first substep in common.
	
We denote by $\bfA_1$ the (skew-)symmetric solution at time $t_1$ of the full problem (\ref{eq:fullEq-mat}), where we consider the initial data  to coincide with the (skew-)symmetric rank-$r$ matrix $\bfY_0$. 
For the local error analysis, the following lemma is needed.

\begin{lemma}\label{lem:sym-aux}
	Let $\bfU_1, \bfA_1$ be defined as above. The following estimate holds:
	$$ \| \bfU_1 \bfU_1^\top \bfA_1 \bfU_1 \bfU_1^\top - \bfA_1 \| \leq 2\vartheta(h, \varepsilon). $$
\end{lemma}

\begin{proof}
	The local error analysis in \cite{KieriLubichWalach} shows that the $r\times n$ matrix 
	$\bfZ=\bfS_1^\mathrm{ps}\bfV_1^{\mathrm{ps},\top}$, where $\bfS_1^\mathrm{ps}$ and $\bfV_1^\mathrm{ps}$ are the matrices computed in the third substep of the projector-splitting algorithm, satisfies
	\begin{equation*}
	\| \bfU_1 \textbf{Z} - \bfA_1 \| \leq \vartheta :=\vartheta(h, \varepsilon) .
	\end{equation*}
The square of the left-hand side can be split into two terms:
	\begin{equation*}
		\begin{aligned}
			\| \bfU_1 \textbf{Z} - \bfA_1 \|^2  
				&= \| \bfU_1 \textbf{Z} - \bfU_1 \bfU_1^\top \bfA_1 + \bfU_1 \bfU_1^\top \bfA_1 - \bfA_1 \| ^2 \\
				&= \| \bfU_1 \bfU_1^\top (\bfU_1 \textbf{Z} - \bfA_1) + (\textbf{I} - \bfU_1 \bfU_1^\top) \bfA_1 \| ^2 \\
				&= \| \bfU_1 \bfU_1^\top (\bfU_1 \textbf{Z} - \bfA_1) \| ^2 +  \|(\textbf{I} - \bfU_1 \bfU_1^\top) \bfA_1 \| ^2 .\\
		\end{aligned}
	\end{equation*} 
Hence, 
	$$ \| \bfU_1 \bfU_1^\top(\bfU_1 \textbf{Z} - \bfA_1) \| ^2 +  \|(\textbf{I} - \bfU_1 \bfU_1^\top) \bfA_1 \| ^2 \leq \vartheta^2 . $$
	From the second term it follows that
	$$ \| \bfU_1 \bfU_1^\top \bfA_1 - \bfA_1 \| \leq \vartheta . $$
	By  the (skew-)symmetry of $\bfA_1$, this implies 
	\begin{equation*}
		\begin{aligned}
			\| \bfU_1 \bfU_1^\top \bfA_1 \bfU_1 \bfU_1^\top - \bfA_1 \| 
			&= \| \bfU_1 \bfU_1^\top \bfA_1 \bfU_1 \bfU_1^\top -\bfU_1 \bfU_1^\top\bfA_1 +\bfU_1 \bfU_1^\top \bfA_1 -\bfA_1 \| \\
			&\leq \| \bfU_1 \bfU_1^\top(\bfA_1 \bfU_1 \bfU_1^\top - \bfA_1) \| + \|\bfU_1 \bfU_1^\top \bfA_1 -\bfA_1 \| \\
			&= \| \bfU_1 \bfU_1^\top(\bfU_1 \bfU_1^\top \bfA_1 - \bfA_1)^\top \| +  \|\bfU_1 \bfU_1^\top \bfA_1 -\bfA_1 \| \\
			&\leq 2\| \bfU_1 \bfU_1^\top \bfA_1 - \bfA_1 \| 			\\
			& \leq 2 \vartheta,
		\end{aligned}
	\end{equation*} 
	which yields the result. \qed
\end{proof}

In the following lemma, we show that the approximation given after one time step is $O(h(h + \varepsilon))$ close to the solution of system (\ref{eq:fullEq-mat}) when the starting values coincide.

\begin{lemma}[Local Error] \label{lem:loc-err} 
The following local error bound holds:
	$$ \| \bfY_1 - \bfA_1 \| \leq h(\hat c_1 \varepsilon  + \hat c_2 h) , $$
	where the constants only depend on $L$ and $B$ and a bound of the step size. In particular, the constants are independent of singular values of the exact or approximate solution. 
\end{lemma}

\begin{proof}
	By the identity $\bfY_1=\bfU_1\bfS_1\bfU_1^\top$ and Lemma~\ref{lem:sym-aux} we have that
	\begin{equation*}
		\begin{aligned}
			\| \bfY_1 - \bfA_1 \| 
			&\leq \| \bfY_1 - \bfU_1 \bfU_1^\top \bfA_1 \bfU_1 \bfU_1^\top \| + \| \bfU_1 \bfU_1^\top \bfA_1 \bfU_1 \bfU_1^\top- \bfA_1 \| \\
			&\leq \| \bfU_1( \bfS_1 - \bfU_1^\top \bfA_1 {\bfU}_1) \bfU_1^\top \| + 2\vartheta \\
			&= \| \bfS_1 - \bfU_1^\top \bfA_1  {\bfU}_1 \|  + 2\vartheta.
		\end{aligned}
	\end{equation*}
	The analysis of the local error thus reduces to estimating $\| \bfS_1 - \bfU_1^\top \bfA_1  {\bfU}_1 \|$.
	To this end, we introduce the following quantity: for $t_0\le t \le t_1$,
	$$ \widetilde\bfS(t) := \bfU_1^\top \bfA(t)  {\bfU}_1 . $$
	We observe that
	\begin{equation*}
		\begin{aligned}
			\bfA(t) 
				&= \bfU_1 \bfU_1^\top \bfA(t) \bfU_1 \bfU_1^\top + \Bigl( \bfA(t) - \bfU_1 \bfU_1^\top \bfA(t) \bfU_1 \bfU_1^\top \Bigr) 
				= \bfU_1 \widetilde\bfS(t) \bfU_1^\top + \textbf{R}(t),
		\end{aligned}
	\end{equation*}
	where $\textbf{R}(t)$ is defined as the term in big brackets.
	From Lemma~\ref{lem:sym-aux} and from the bound $B$ of $\bfF$, which yields for $t_0 \le t \le t_1$
	$$
	\| \bfA(t) - \bfA(t_1) \| \le \int_{t_0}^{t_1} \| \dt \bfA(s) \|\, ds =
	\int_{t_0}^{t_1} \| \bfF(s,\bfA(s)) \| \, ds \le Bh,
	$$
	we conclude that the remainder term is bounded by
	$$ 
		\| \textbf{R}(t) \| 
			\leq \|\textbf{R}(t) - \textbf{R}(t_1) \| + \| \textbf{R}(t_1) \|
			\leq 2Bh + 2\vartheta. 
	$$
	This yields that $\bfF(t, \bfA(t))$ can be written as
	\begin{equation*}
		\begin{aligned}
			\bfF(t, \bfA(t)) 
				&= \bfF(t, \bfU_1 \widetilde\bfS(t) \bfU_1^\top + \textbf{R}(t) ) \\
				&= \bfF(t, \bfU_1 \widetilde\bfS(t) \bfU_1^\top) + \bfD(t),
		\end{aligned}
	\end{equation*}
	where the defect 
	$$
	 \bfD(t) := \bfF(t, \bfU_1 \widetilde\bfS(t) \bfU_1^\top + \textbf{R}(t)) - \bfF(t, \bfU_1 \widetilde\bfS(t) \bfU_1^\top)
	$$
	is bounded via the Lipschitz continuity of $\bfF$ as
	$$
	\|  \bfD(t) \| \le L \| \textbf{R}(t) \| \le 2L (Bh + \vartheta).
	$$
	We now compare the two differential equations
	\begin{equation*}
		\begin{aligned}
			&\dot{\widetilde\bfS}(t) = \bfU_1^\top \bfF(t, \bfU_1 \widetilde\bfS(t) \bfU_1^\top)  {\bfU}_1 + \bfU_1^\top \bfD(t)  {\bfU}_1, 
			\qquad
			&\widetilde\bfS(t_0) = \bfU_1^\top \bfY_0  {\bfU}_1,\\
			&\dot{\bfS}(t) = \bfU_1^\top \bfF(t, \bfU_1 \bfS(t) \bfU_1^\top)  {\bfU}_1, 
			\qquad
			&\bfS(t_0) = \bfU_1^\top \bfY_0  {\bfU}_1.
		\end{aligned}
	\end{equation*}
	By construction, the solution of the first differential equation at time $t_1$ is  $ \widetilde \bfS(t_1) = \bfU_1^\top \bfA_1  {\bfU}_1$. The solution of the second differential equation is $\bfS_1$  as given by the second substep of the (skew-)symmetric integrator.
	We  now apply the Gronwall inequality to the previous system and obtain
	$$
		\|  \bfS_1 - \bfU_1^\top \bfA_1  {\bfU}_1 \| 
			\leq \int_{t_0}^{t_1} e^{L(t_1-s)} \, \| \bfD(s) \| \, ds
			\leq e^{Lh} \,2L (Bh + \vartheta) h.
	$$
	The result now follows using the definition of $\vartheta$. \qed
\end{proof}

Thanks to the Lipschitz continuity of the function $\bfF$, we conclude the proof of Theorem \ref{thm:sym-robust} from the local to the global errors by the standard argument of Lady Windermere's fan \cite[Section II.3]{HairerNorsettWanner:ODE_BOOK1}.

\section{General tensors: recap of the projector-splitting integrator for the dynamical low-rank approximation by Tucker tensors}
The objective is to approximate time-depend\-ent tensors\footnote{In view of the applications in quantum dynamics, we here consider tensors with complex entries.} $A(t)\in \C^{n_1\times\dots\times n_d}$ for $0\le t \le T$ by tensors  $Y(t)$ of multilinear rank
$\bfr=(r_1,\dots,r_d)$, with $r_i \ll n_i$. (We recall that $r_i$ is the rank of the $i$th matricization $\mathbf{Mat}_i(Y)\in \C^{n_i\times n_i'}$ with $n_i'=\prod_{j\ne i} n_j$, which aligns  all entries of $Y$ with $i$th index $k$ in the $k$th row. The retensorization is denoted by {\it Ten}$_i(\cdot)$, such that {\it Ten}$_i(\mathbf{Mat}_i(Y)) = Y$.)

The tensors $A(t)$ may be given explicitly or they are the unknown solution to a tensor differential equation
(with right-hand side function $F:\R\times \C^{n_1\times\dots\times n_d} \to \C^{n_1\times\dots\times n_d}$)
\begin{equation} \label{eq:fullEq-ten}
\dt{A}(t) = F(t, A(t)), 
\qquad
A(0) = A_0 .
\end{equation}
Dynamical low-rank approximation as presented in  \cite{KochLubich10} determines $Y(t)$ as the solution of the projected matrix
differential equation,
with a projection $\P(Y)$ onto the tangent space $T_Y \calM_\bfr$ of the manifold of tensors of multilinear rank $\bfr$ at $Y\in\calM_\bfr$,
\begin{equation} \label{eq:projEq-ten}
\dt{Y}(t) = \P(Y(t)) F(t, Y(t)),
\qquad
Y(t_0) = Y_0,
\end{equation}
where $Y_0$ is a rank-$\bfr$ approximation to $A_0$. (Here, $F(t,Y) = \dt A(t)$ if $A(t)$ is given explicitly.)
Tensors $Y(t)$ of multilinear rank $\bfr$ are represented non-uniquely in the Tucker form \cite{DeLauthawer:HOSVD} (using here the multilinear notation of \cite{KoldaBader:TensorDec})
\begin{equation} \label{Tucker}
	Y(t) = C(t) \bigtimes_{i=1}^d \bfU_i(t) ,
\end{equation}
where the core tensor $C(t) \in \mathbb{C}^{r_1 \times \dots \times r_d }$ is of full multi-linear rank and the basis matrices $\bfU_i \in \mathbb{C}^{n \times r_i}$ have orthonormal columns.
We choose the tangent space projection $\P(Y)$ as the orthogonal projection onto $T_Y(\calM_r)$ with respect to the Euclidean inner product $\langle A,B \rangle = \textbf{vec}(A)^* \textbf{vec}(B)$, where $\textbf{vec}(A)$ is a vectorization of $A$.
Then, $\P(\bfY)$ is given as an alternating sum of $2d-1$ subprojections \cite{Lubich:MCTDH}, and like in the matrix case, a projector-splitting integrator with favourable properties can be formulated and efficiently implemented. The matrix projector-splitting integrator proposed in Section 2.1 has been successfully extended to the Tucker tensor format  in different algorithmic versions in \cite{Lubich:MCTDH} and \cite{LubichVandWalach}. It is shown in \cite[Section 6]{LubichVandWalach} that the proposed Tucker integrators are mathematically equivalent.  The algorithm runs through the modes $i=1,\dots,d$ and solves differential equations for matrices of the dimension of the slim basis matrices and for the core tensor in alternation with orthogonal decompositions of slim matrices. We refer also to
\cite{KlossBL17,BonfantiB18} for the formulation and implementation of this algorithm in the context of the MCTDH method \cite{MeyerGW09} of molecular quantum dynamics in the chemical physics literature.

Moreover, the Tucker integrator has been proved in \cite{LubichVandWalach} to satisfy analogous properties to the matrix projector-splitting integrator: the exactness property and  the robust convergence in the presence of small singular values of matricizations of the core tensor. We refer to \cite[Theorems 4.1 and 5.1]{LubichVandWalach} for the precise formulation, which is very similar to the matrix case.

\section{Symmetric and anti-symmetric tensors: a structure-preserving integrator for dynamical low-rank approximation}
\label{sec:sym-ten}
A tensor $A=(a_{i_1,\dots,i_d})\in \C^{n\times\dots\times n}$ is {\it symmetric} if for every permutation $\sigma\in S(d)$,
$$
a_{i_{\sigma(1)},\dots, i_{\sigma(d)}} = a_{i_1,\dots,i_d},
$$
and $A$ is {\it anti-symmetric} if for every permutation $\sigma\in S(d)$,
$$
a_{i_{\sigma(1)},\dots, i_{\sigma(d)}} = (-1)^{\mathrm{sign}(\sigma)}a_{i_1,\dots,i_d}.
$$

It follows from \cite{DeLauthawer:HOSVD} and \cite{Hackbusch:SymTenRepr} that a symmetric/anti-symmetric tensor $Y\in \C^{n\times\dots\times n}$ of multi-linear rank $ \textbf{r} = (r, \dots, r)$  admits a Tucker decomposition
$$ 	
Y = C \bigtimes_{i=1}^d \bfU,  
$$
where the core tensor $C\in \C^{r\times\dots\times r}$ is symmetric/anti-symmetric of full rank $\bfr$ and the basis matrix 
$\bfU \in \C^{n\times r}$
is the same for all indices.

We assume that the right-hand side function in \eqref{eq:fullEq-ten} is such that
\begin{equation} \label{F-sym-ten}
\text{$F(t,Y)$ is (anti-)symmetric whenever $Y$ is (anti-)symmetric.}
\end{equation}
Like \eqref{F-sym} in the matrix case, this ensures that the solutions to the tensor differential equation \eqref{eq:fullEq-ten} and the projected differential equation \eqref{eq:projEq-ten} are (anti-)symmetric provided the initial tensors are (anti-)symmetric. As we noted already in the matrix case, the projector-splitting Tucker integrator does not preserve (anti-)symmetry.

\subsection{(Anti-)symmetry preserving Tucker integrator}
The numerical integrator defined in Section \ref{sec:Sym} for the matrix case extends in a natural way to the Tucker tensor format. The first substep, which updates the basis matrix $\bfU$, is identical to the first substep of the general Tucker integrator in \cite{LubichVandWalach,Lubich:MCTDH}. The second substep is a Galerkin method with the updated basis and determines the updated (anti-)symmetric core tensor.

Given the (anti-)symmetric tensor $Y_0 = C_0 \bigtimes_{i=1}^{d} \bfU_0 $, we compute the (anti-)sym\-metric approximation \- $Y_1 = C_1 \bigtimes_{i=1}^{d} \bfU_1 $ at time $t_1 = t_0 + h$ as follows: \\

\hfill \break
\begin{algorithm}[H]
	\caption{One time step of the (anti-)symmetry preserving Tucker integrator}
	\KwData{Tucker tensor $Y_0 = C_0 \bigtimes_{i=1}^d \bfU_0$, $F(t,Y)$, $t_0$, $t_1$}
	\KwResult{Tucker tensor $Y_1 = C_1 \bigtimes_{i=1}^d \bfU_1$}
	\Begin{
		Matricize the core tensor $C_0$ in the first mode.

		Perform a QR-factorization:
			$$ \text{\textbf{Mat}}_1(C_0)^\top = \textbf{Q}_0 \bfS_0^\top ,$$
			where $ \textbf{Q}_0 \in \C^{r^{d-1}\times r}$ has orthonormal columns.
			Define
			$$ \bfV_0^\top = 
					\textbf{Q}_0^\top \bigotimes_{i = 2}^d \bfU_0^\top \ .
			$$

		Integrate from $t=t_0$ to $t_1$ the $n\times r$ matrix differential equation
		\begin{equation*}
			\dot{\textbf{K}}(t) = 
				\text{\textbf{Mat}}_1( 
						F(t, \text{\it Ten}_1( \textbf{K}(t)\bfV_0^\top))
					) \overline{\bfV}_0,
			\qquad
			\textbf{K}(t_0) = \bfU_0 \bfS_0.
		\end{equation*}

		Compute the QR-factorization 
		$ \textbf{K}(t_1) = \bfU_1 \textbf{R}$ \ .
		
		\
		
	   Integrate from $t=t_0$ to $t_1$ the $r\times\dots\times r$ ($d$ times) tensor equation
		\begin{align*}
			&\dot{C}(t) = F \left( t, C(t) \bigtimes_{i=1}^d \bfU_1 \right) \bigtimes_{i=1}^d \bfU_1^*,
			\\
			&C(t_0) = Y_0 \bigtimes_{i=1}^d \bfU_1^* = C_0 \bigtimes_{i=1}^d (\bfU_1^*\bfU_0). 
		\end{align*}
		
		Set $C_1=C(t_1)$.
	}
\end{algorithm}
\vspace{10pt}
%
To continue, we take $Y_1$ as the starting value for the next step.
%
%
%

\subsection{Exactness property of the (anti-)symmetric Tucker integrator} The following result extends
the exactness results of Theorem~\ref{thm:sym-exact} and \cite[Theorem 4.1]{LubichVandWalach} to (anti-)symmetric tensors. 
\begin{theorem}[Exactness property]
\label{thm:sym-exact-ten}
	Let $A(t) \in \C^{n \times\dots\times n}$ be (anti-)symmetric and of multilinear rank~$(r,\dots,r)$  for $t_0 \leq t \leq t_1$,
	so that $A(t)=C(t)\bigtimes_{i=1}^d \bfU(t)$, where the $n\times r$ basis matrix $\bfU$ has orthonormal columns. Moreover, assume  that the $r\times r$ matrix $ \bfU(t_1)^* \bfU(t_0)$ is invertible.
	With $Y_0=A(t_0)$, the (anti-)symmetric Tucker integrator for 
	$\dt Y(t)=\P(Y(t))\dt A(t)$ is then exact: $Y_1 = A(t_1)$.
\end{theorem}
%

\begin{proof}
	The projector-splitting Tucker integrator and the (anti-)symmetric integrator have the same first substep. Let $\bfU_1 \in \mathbb{R}^{n \times r}$ be the basis matrix with orthonormal columns computed in the first substep. Due to the exactness of the projector-splitting Tucker integrator as shown by \cite[Theorem 4.1]{LubichVandWalach} we have that $A(t_1)$ has  the (anti)-symmetric Tucker representation
$$
A(t_1)= \widehat C_1  \bigtimes_{i=1}^{d} \bfU_1
$$
for some (anti-)symmetric core tensor $ \widehat C_1 \in \C^{r\times\dots\times r}$. 
		Using the rule $A \times_i \bfV \times_i \textbf{W} = A \times_i (\textbf{WV})$, this implies that
		$$ A(t_1) \times_i (\bfU_1 \bfU_1^*) = A(t_1),\qquad  i = 1, \dots , d \ .$$
	 With $Y_0 = A(t_0)$  we obtain from the second substep of the algorithm
		\begin{align*} Y_1 &= C_1 \bigtimes_{i=1}^d \bfU_1  =
		\Bigl( Y_0 \bigtimes_{i=1}^d  \bfU_1 ^* + (A(t_1)-A(t_0)) \bigtimes_{i=1}^d  \bfU_1^* \Bigr) \bigtimes_{i=1}^d \bfU_1
		\\
		&= \Bigl( A(t_1)\bigtimes_{i=1}^d  \bfU_1^* \Bigr) \bigtimes_{i=1}^d \bfU_1 
		= A(t_1) \bigtimes_{i=1}^d (\bfU_1 \bfU_1^*) = A(t_1),
		\end{align*}		
		which proves the exactness. \qed
 \end{proof}

\subsection{Robustness to small singular values} The robust error bounds from Theorem~\ref{thm:sym-robust} and \cite[Theorem~5.1]{LubichVandWalach} extend to the (anti)\-symmetric Tucker integrator as follows. The norm $\|B\|$ of a tensor $B$ used here is the Euclidean norm of the entries of $B$.

\begin{theorem}[{Robust error bound}]
\label{thm:sym-robust-ten}
	Let $A(t)$ denote the (anti-)symmetric solution of the tensor differential equation \eqref{eq:fullEq-ten} with $F$ satisfying \eqref{F-sym-ten}. Assume  the following:
	\begin{enumerate}
		\item 
		$F$ is Lipschitz-continuous and bounded.		
		\item
		The non-tangential part of $F(t, Y)$ is $\varepsilon$-small:
		$$
		\| (I - \P(Y)) F(t, Y) \| \le \eps
		$$
		for all $Y$ of multilinear rank  $(r,\dots,r)$ in a neighbourhood of $A(t)$ and $0\le t \le T$.
		\item
		The error in the initial value is $\delta$-small:
		$$
		\| Y_0 - A_0 \| \le \delta.
		$$
	\end{enumerate}	
	Let $Y_n$ denote the (anti-)symmetric approximation of multinear rank $(r,\dots,r)$ to $A(t_n)$ at $t_n=nh$ obtained after n steps of the (anti-)symmetric Tucker integrator with step-size $h>0$.
	Then, the error satisfies for all $n$ with $t_n =  nh \leq T$
	$$ \| Y_n - A(t_n) \| \leq c_0\delta + c_1 \varepsilon + c_2 h ,$$	
	where the constants $c_i$ only depend on the Lipschitz constant $L$ and bound $B$ of $F$, on~$T$, and on the dimension $d$. In particular, the constants are independent of singular values of matricizations of the exact or approximate solution. 
\end{theorem}

It can be further shown that an inexact solution of the matrix differential equations in the  integrator
leads to an additional error that is bounded in terms of the local errors in the inexact substeps, again with constants that do not depend on small singular values.
 
The proof of Theorem \ref{thm:sym-robust-ten} proceeds similar to the proof of Theorem~\ref{thm:sym-robust} for the (skew)-symmetric matrix case. We begin with a key lemma and then analyze  the local error produced after one time step, comparing the numerical solution with the exact solution that starts from the same initial value $A_0=Y_0$. We denote the value of this solution at $t_1$ by $A_1$. The basis matrix computed in the first substep of the integrator is denoted by $\bfU_1$.

\begin{lemma} \label{lem:aux-ten}
	The following estimate holds:
	$$ || A_1 \bigtimes_{i=1}^d \bfU_1 \bfU_1^* - A_1 || \leq  c \,h(BLh+\varepsilon), $$
	where $c$ only depends on $d$ and a bound for $hL$.
\end{lemma}

\begin{proof}
		The error bound of \cite[Theorem 5.1]{LubichVandWalach} shows that there exists $Z \in \mathbb{C}^{r \times n\times \dots \times n}$ such that
		\begin{equation*}
		 \| Z \times_1 \bfU_1 - A_1 \| \leq c_* h(BLh+\varepsilon) =: \vartheta.
		\end{equation*}
		We observe that
		\begin{equation*}
			\begin{aligned}
			\| Z \times_1 \bfU_1 - A_1 \| 
				= \| \textbf{Mat}_1(Z \times_1 \bfU_1 - A_1)\| 
				= \| \bfU_1 \textbf{Mat}_1(Z) - \textbf{Mat}_1(A_1) \| .
			\end{aligned}
		\end{equation*}
		As in the matrix case we obtain 
		$$ \| \bfU_1 \bfU_1^* \textbf{Mat}_1(A_1) - \textbf{Mat}_1(A_1) \| \leq \vartheta . $$
		Thanks to (anti-)symmetry we have 
		$$ \| \bfU_1 \bfU_1^* \textbf{Mat}_1(A_1) - \textbf{Mat}_1(A_1) \|
			 = \| \bfU_1 \bfU_1^* \textbf{Mat}_i(A_1) - \textbf{Mat}_i(A_1)\|,
			 \quad \  i=1, \dots ,d , $$
which yields
		$$ \|  A_1 \times_i \bfU_1 \bfU_1^* - A_1 \| \leq \vartheta, \qquad  \  i=1, \dots, d .$$
		To conclude, we observe 
		\begin{equation*}
			\begin{aligned}
				&\|  A_1 \bigtimes_{i=1}^d \bfU_1 \bfU_1^* - A_1 \| 
				\\
				&\leq \|  A_1 \bigtimes_{i=1}^d \bfU_1 \bfU_1^* -A_1 \bigtimes_{i=1}^{d-1} \bfU_1 \bfU_1^* 
				+A_1 \bigtimes_{i=1}^{d-1} \bfU_1 \bfU_1^* - A_1 \| \\
				&\leq \|  A_1 \bigtimes_{i=1}^d \bfU_1 \bfU_1^* -A_1 \bigtimes_{i=1}^{d-1} \bfU_1 \bfU_1^* \|
					+ \| A_1 \bigtimes_{i=1}^{d-1} \bfU_1 \bfU_1^* - A_1 \| \\
				&\leq \| (A_1 \times_d \bfU_1 \bfU_1^* - A_1) \bigtimes_{i=1}^{d-1} \bfU_1 \bfU_1^* \| 
					+ \| A_1 \bigtimes_{i=1}^{d-1} \bfU_1 \bfU_1^* - A_1 \| \\
				&\leq \| A_1 \times_d \bfU_1 \bfU_1^* - A_1 \| 
					+ \| A_1 \bigtimes_{i=1}^{d-1} \bfU_1 \bfU_1^* - A_1 \| \\
				&\leq \vartheta + \| A_1 \bigtimes_{i=1}^{d-1} \bfU_1 \bfU_1^* - A_1 \| ,
			\end{aligned}
		\end{equation*}
		and the result follows by an iteration of this argument. \qed
\end{proof}

We are now in the  position to analyse the local error produced after one time step of the integrator. 

\begin{lemma}[Local error] The error of the (anti-)symmetric Tucker integrator after one time step satisfies 
$$ \| Y_1 - A_1 \| \leq \hat c \,h(BLh+\eps) , $$
	where $\hat c$ only depends on $d$ and a bound of $hL$. In particular, the constant is independent of singular values of the exact or approximate solution. 
\end{lemma}

\begin{proof}
		By construction of the algorithm and by Lemma~\ref{lem:aux-ten} we have 
		\begin{equation*}
		\begin{aligned}
		\| Y_1 - A_1 \|
		&\leq \| Y_1 - A_1 \bigtimes_{i=1}^d \bfU_1 \bfU_1^* \| + \| A_1 \bigtimes_{i=1}^d  \bfU_1 \bfU_1^* - A_1 \| \\
		&\leq \| C_1 \bigtimes_{i=1}^d \bfU_1 - A_1 \bigtimes_{i=1}^d \bfU_1 \bfU_1^*  \| + c\vartheta \\
		&\leq \| C_1 \bigtimes_{i=1}^d \bfU_1 - (A_1 \bigtimes_{i=1}^d \bfU_1^*) \bigtimes_{i=1}^d \bfU_1  \| + c\vartheta \\
		&\leq \| ( C_1 - A_1 \bigtimes_{i=1}^d \bfU_1^*) \bigtimes_{i=1}^d \bfU_1 \|  + c \vartheta \\
		&\leq \| C_1 - A_1 \bigtimes_{i=1}^d \bfU_1^* \| + c \vartheta.
		\end{aligned}
		\end{equation*}
		The problem reduces to estimating $\| C_1 - A_1 \bigtimes_{i=1}^d \bfU_1^* \|$. 
		We introduce the  tensor 
		$$ \widetilde C(t) := A(t) \bigtimes_{i=1}^d \bfU_1^* , 
		$$
		which satisfies 
		$$
		\dt{\widetilde C}(t) = F(t, A(t)) \bigtimes_{i=1}^d \bfU_1^* , 
		\qquad
		\widetilde C(t_0) = Y_0 \bigtimes_{i=1}^d \bfU_1^*.
		$$
		In the same way as in the proof of Lemma~\ref{lem:loc-err} (replacing $\bfS$ by $C$) this is compared with the differential 
		equation for $C(t)$ in the second substep of the integrator. This yields the stated result. \qed
%
\end{proof}

Using the Lipschitz continuity of the function $F$, we conclude the proof of Theorem~\ref{thm:sym-robust-ten} from the local to the global errors by the standard argument of Lady Windermere's fan \cite[Section II.3]{HairerNorsettWanner:ODE_BOOK1}.

\section{Numerical Experiments}
In this section we show results of various numerical experiments. The computations were done using Matlab R2017a software with Tensor Toolbox package v2.6 \cite{TTB_Software}  and TensorLab package v3.0 \cite{vervliet2016tensorlab}. 

\subsection{Addition of symmetric tensors: a computationally inexpensive retraction}
Let $A \in \mathbb{C}^{n \times \dots \times n}$ be a symmetric tensor of multi-linear rank $\textbf{r}=(r, \dots, r)$ and let  $B \in \mathbb{C}^{n \times \dots \times n}$. We consider the addition of two given tensors,
$$ C = A + B,$$
where $C \in \mathbb{C}^{n \times \dots \times n}$ is not necessarily of low rank and we want to compute a symmetric rank-$(r,\dots,r)$ approximation. Such a retraction is typically  required in optimization problems on low-rank manifolds and needs to be computed in each iterative step of a descent algorithm. 
%
%
The approach considered here consists of  reformulating the addition problem as the solution of the following differential equation at time $t=1$:  
$$ \dot{C}(t) = B, \quad C(0) = A.$$

We will compare the solution obtained by computing the full addition and retracting to the manifold of symmetric tensors of multilinear rank $(r,\dots,r)$ with the one obtained from the application of the symmetric low-rank tensor integrator. The advantage of the latter method is that the approximation is built inside the manifold, so that no truncation to rank $r$ is needed. \\

For our numerical example, we initialize A as a symmetric random Tucker tensor of size $100 \times 100 \times 100$ and multi-linear rank $\textbf{r} = (10, 10, 10)$; we take B as an element in the tangent space of the symmetric rank-$\textbf{r}$ tensor manifold at $A$. We compare the dynamical low rank approximation  $Y_1$ generated by the algorithm introduced in Section~\ref{sec:sym-ten} with a low rank symmetric retraction \cite{de2000best,regalia2013monotonically} of the full solution denoted by $X$. For the last part, we use the built-in $\texttt{tucker\_als}$ and $\texttt{tucker\_sym}$ functions of the Tensor Toolbox Package.

\begin{center}
	\includegraphics[width=\textwidth]{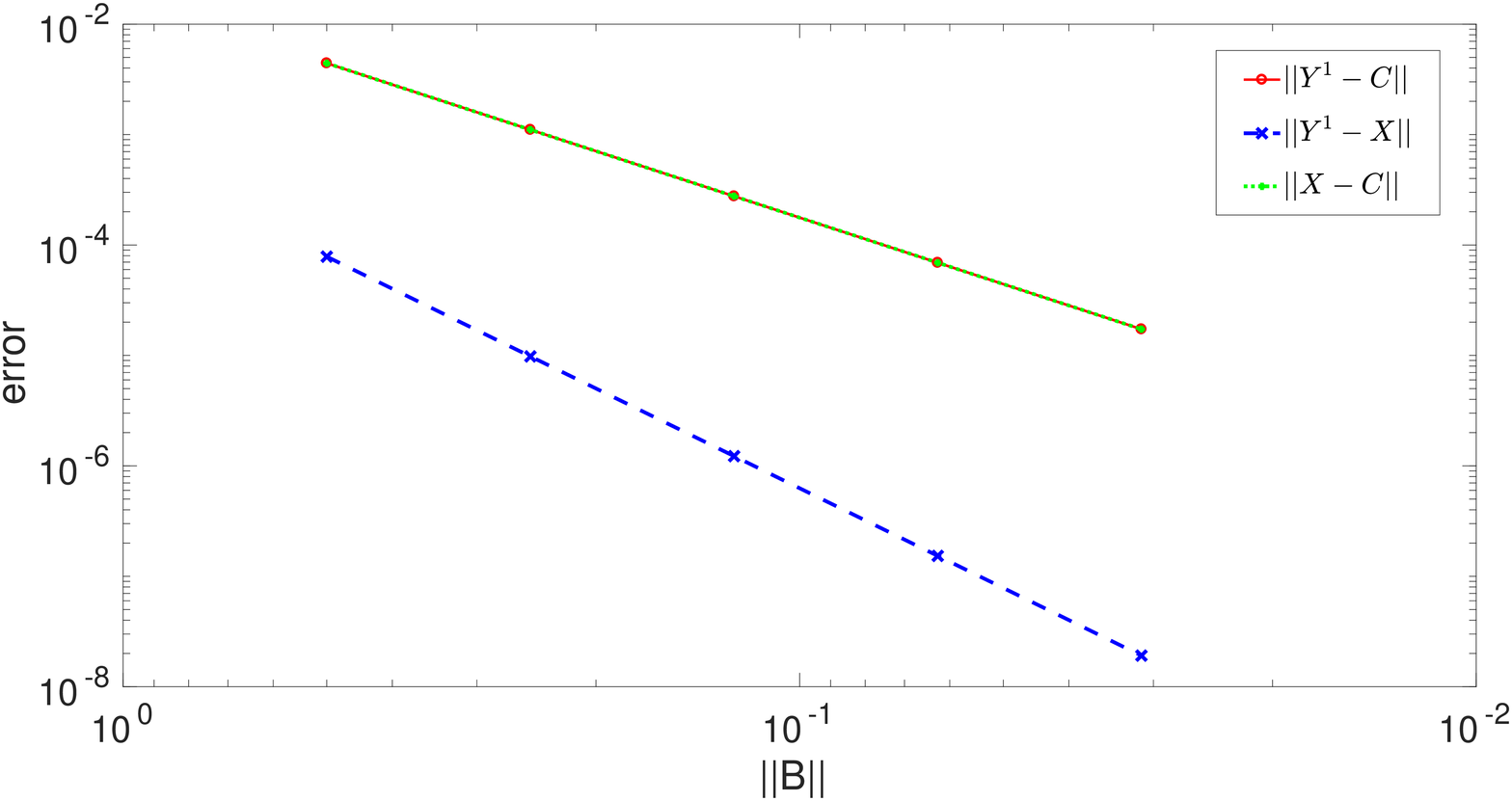}
\end{center}

We observe that the approximation $Y_1$ shows the correct behavior, at reduced computational cost.
Decreasing the norm of the tensor $B$ decreases the approximation error as expected, proportional to $\| B \|^2$.

\subsection{Robustness with respect to small singular values}

We present two numerical examples and show robustness of the proposed symmetric integrator in the presence of small singular values. For the sake of presentation we consider the matrix case. Analogous examples can be implemented for Tucker tensors, and similar results are obtained.

In the first example, the time-dependent matrix is given explicitly as
$$ \textbf{A}(t) = \big( e^{t\textbf{W}} \big) e^t \textbf{D} \big( e^{t\textbf{W}} \big)^\top, \quad 0 \leq t \leq 1 \ .$$
The matrix $\textbf{D} \in \R^{N \times N}$ is diagonal with elements $d_j = 2^{-j}$ and the matrix $\textbf{W} \in \R^{N \times N}$ is skew-symmetric and randomly generated. We choose $N=100$ and final time $T=1$.  We compare the symmetric low-rank integrator presented in Section~\ref{sec:Sym} with a numerical solution obtained with a 4-th order explicit Runge-Kutta method applied to the system of differential equations for dynamical low-rank approximation as derived in \cite{KochLubich07}. 

The numerical results for different ranks are shown in Figure \ref{fig:explicitmatrix}. In contrast to the Runge--Kutta method, the proposed symmetric low-rank integrator does not suffer of a step-size restriction in the presence of small singular values. 
\\ \\
In the second example, we integrate the Lyapunov  matrix differential equation (cf.~\cite{mena2018numerical})
$$ \dot{\textbf{X}}(t) = \textbf{AX}(t) + \textbf{X}(t)\textbf{A}^\top + \textbf{Q}, \quad \textbf{X}(0)=\textbf{U}_0 \textbf{S}_0 \textbf{U}_0^\top \ . $$
Here, we choose $\textbf{A} = \texttt{tridiag}(-1,2,-1) \otimes \textbf{I} + \textbf{I} \otimes \texttt{tridiag}(-1,2,-1) \in \R^{N \times N}$ as a discrete Laplacian. 
The positive definite matrix $\textbf{Q} \in \R^{N \times N}$ has rank 5 and is randomly generated. The orthonormal matrix $\textbf{U}_0 \in \R^{N \times N}$ is randomly generated and $\textbf{S}_0 \in \R^{N \times N}$ is of rank 1 with only one non-zero element, $s_{11} = 1 $.  

The reference solution and the linear subproblems appearing in the definition of the symmetric low-rank  integrator have been solved with the Matlab solver \texttt{ode45} and stringent tolerance parameters \textsc{\{'RelTol', 1e-10, 'AbsTol', 1e-14\} }.  We choose $N=100$ and final time $T=0.1$. The singular values of the reference solution and the absolute errors $\| Y_n -A(t_n) \|_F$ at time $t_n=T$ of the approximate solutions for different ranks calculated with different step-sizes are shown in Figure~\ref{fig:lyapunov}.

\begin{figure}
	\includegraphics[height=5cm, width=\textwidth]{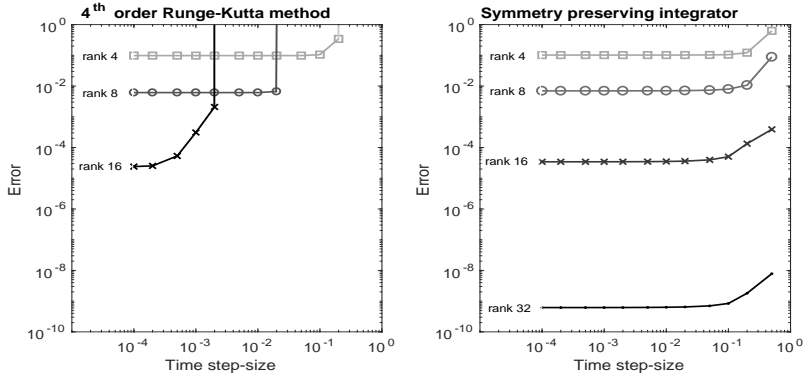}
	\caption{Comparison of the explicit Runge Kutta method (left) and the proposed symmetry-preserving integrator (right) for
		different approximation ranks and step sizes in the case of a given time-dependent symmetric matrix. }
	\label{fig:explicitmatrix}
\end{figure}

\begin{figure}
	\includegraphics[height=5cm, width=\textwidth]{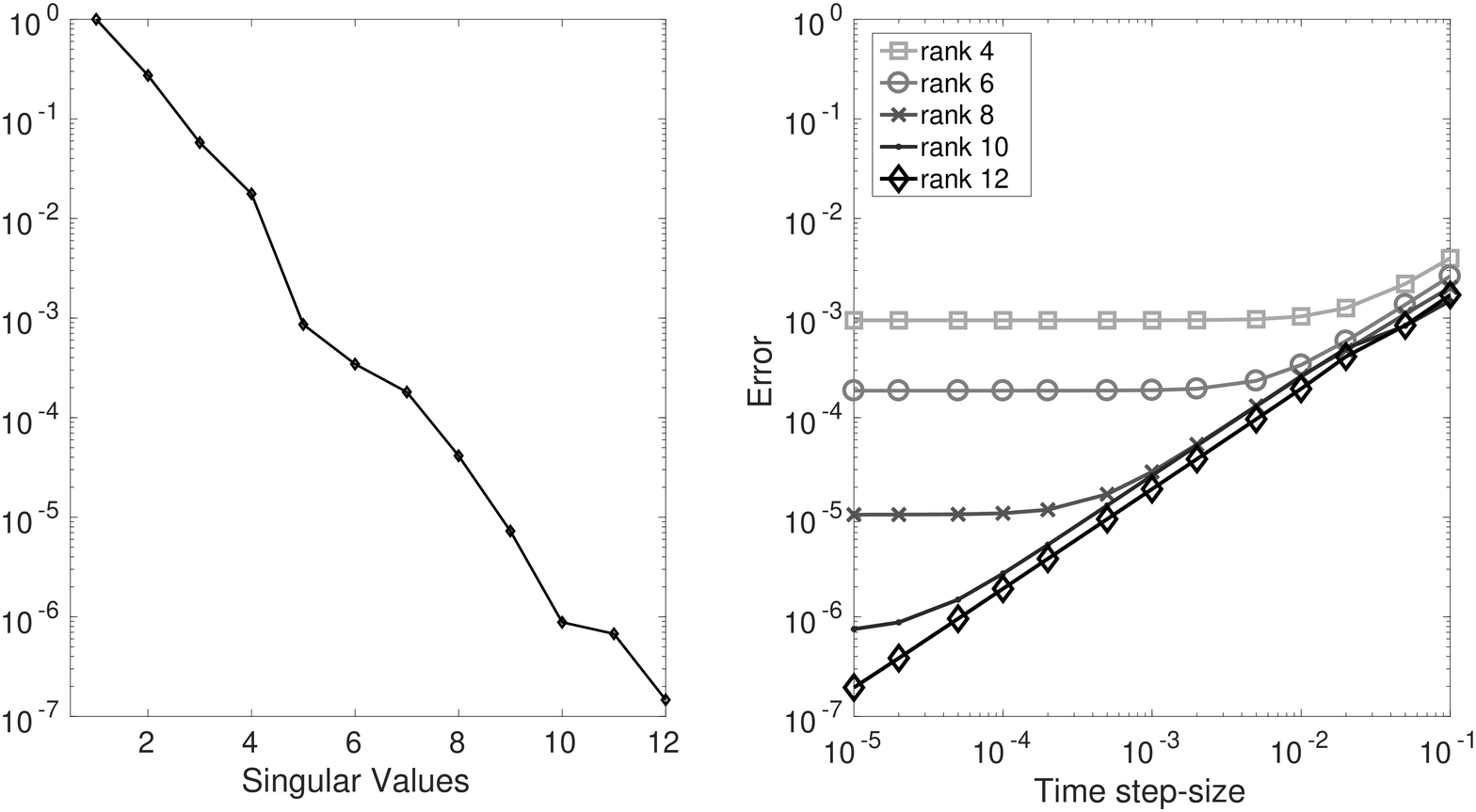}
	\caption{First twelve singular values of the reference solution at time $T=0.1$ and approximation errors for different ranks at different step-sizes for the Lyapunov matrix differential equation.}
	\label{fig:lyapunov}
\end{figure}

\subsection{Ground state of a fermionic multi-particle system}
A natural field of application of the (anti-)symmetric low-rank algorithms is in quantum dynamics of systems  of fermions or bosons, which is described by anti-symmetric or symmetric multivariate wave functions, respectively. In the MCTDHF and MCTDHB methods \cite{AlonSC08,Caillat:MCTDHF}, the approximate wave function is sought for in the form of a low-rank Tucker tensor that is 
anti-symmetric and symmetric, respectively.  It approximates the huge tensor of coefficients with respect to some fixed spatial basis, such as a Fourier basis. The (anti-)symmetric low-rank time integrator proposed in this paper appears ideally suited as a numerical integrator for such problems.

As a first illustration of the approach, we apply the method for the calculation of the ground state of a system of $d$ fermions in 1 space dimension.
We calculate the ground state as the solution for large time of the imaginary-time Schr\"odinger equation
\begin{equation} \label{eq:GroundState}
	\partial_t \psi = -\mathcal{H} \psi, \quad \psi(t_0) = \psi_0.
\end{equation}
We consider the Hamiltonian given by 
	$$ 
	\mathcal{H} := \sum_{l=1}^{d} \Bigl(
	- \tfrac{1}{2} \partial^2_l +
	V(x_l) -
	\sum_{k=l+1}^{d}
	V(x_l - x_k) \Bigr),
	$$
where $x_l \in \R$ represents the position of the $l$-th particle and we choose the torsion potential
$$ 
	V(x) = 1 - \cos(x) \ .
$$
Choosing a collocation method with a tensor Fourier basis set (with $K$ basis functions per particle)
for approximating the anti-symmetric wave function leads to a huge tensor differential equation \eqref{eq:fullEq-ten}, where a low-rank approximation to the anti-symmetric $d$-dimensional tensor $Y(t)\in \C^{K\times\dots\times K}$ is to be computed.  This is done with a variational splitting method. The stiffness introduced by the Laplacian will be handled with a split-step Fourier method \cite{lubich2008quantum} while the two-particle interaction is treated with the anti-symmetric low-rank integrator of Section~\ref{sec:sym-ten}. 

\noindent
We introduce the space discretization 
$$x_j = \frac{2 \pi j }{K}, \quad j=-K/2, \dots, K/2-1 \ .$$ 
Let $Y(t)$ be a time-dependent tensor defined element-wise by
$$ Y_{j_1, j_2, \dots, j_d}(t) = \psi(t, x_{j_1}, \dots, x_{j_d}) \ .$$
The fermionic property of the system implies that the tensor $Y(t)$ is anti-symmetric.
Denoting $\mathcal{F}_K$ the Fourier matrix, we define
\begin{align*}
& \textbf{D} := \mathcal{F}_K^{-1} \text{diag} \{ \tfrac{1}{2} j^2 \} \mathcal{F}_K ,
\\
& \textbf{V}_\text{cos} := \text{diag} \{ \cos(x_j) \},
\\
& \textbf{V}_\text{sin} := \text{diag} \{ \sin(x_j) \} \ .
\end{align*}
The Fourier collocation space discretization of (\ref{eq:GroundState}) is equivalent to the system
\begin{equation}
	\label{eq:DiscreteGroundState}
	\dot{Y} = -H[Y], \qquad Y(t_0) = C_0 \bigtimes_{i=1}^d \textbf{U}_0 .
\end{equation}
Using the trigonometric equality
$\cos(x-y) = \cos(x)\cos(y) + \sin(x)\sin(y)$,
the linear operator $H$ can be written in a multi-linear product form as
$$ 
	H[Y] =  \frac{3d - d^2}{2} Y +
			\sum_{l=1}^{d} Y \times_l  \textbf{D} - 
			Y \times_l \textbf{V}_{\text{cos}} + 
			\sum_{k=l+1}^{d} 
			Y \times_l \textbf{V}_{\cos} \times_k \textbf{V}_{\cos} + 
			Y \times_l \textbf{V}_{\sin} \times_k \textbf{V}_{\sin}  
			\ .
$$
In order to remove the stiffness  introduced by the Laplacian, we split \eqref{eq:DiscreteGroundState} in $(d+2)$ sub-problems.
The solution of the first sub-problem at time $t_1 = t_0 + h$ is obtained updating the core tensor $C_0$ in the initial data,
$$ \widetilde{C}_0 = \exp \Big( -h \frac{3d - d^2}{2} \Big) C_0 \ .$$
Afterwards, we start considering the equation 
$$ \dot{Y}_I = -Y_I \times_1 \textbf{D} + Y_I \times_1 \textbf{V}_{\cos}, \quad Y_I(t_0) = \widetilde{C}_0 \times \textbf{U}_0 \ .$$
We matricize in the first mode, 
$$  \textbf{Mat}_1( \dot{Y}_I) = -\textbf{D} \,\textbf{Mat}_1(Y_I) + \textbf{V}_{\cos}\textbf{Mat}_1(Y_I)$$
with initial data,
$$ \textbf{Mat}_1(Y_I(t_0)) = \textbf{Mat}_1(Y_0) = \textbf{U}_0 \textbf{Mat}_1( \widetilde{C}_0 \bigtimes_{i=2}^d \textbf{U}_0) \ . $$
The solution can now be computed with the 1-dimensional split-step Fourier method. 
Denoting
$$
	\widetilde{\textbf{U}}_0= 
		e^{+\frac{h}{2} \textbf{V}_{\cos}}
		\mathcal{F}_K^{-1}
		e^{-h\textbf{T}}
		\mathcal{F}_K
		e^{+\frac{h}{2} \textbf{V}_{\cos}}
		\textbf{U}_0 \ ,
		\quad
		\textbf{T} = \mathcal{F}_K \textbf{D} \mathcal{F}_K^{-1} 
$$
and tensorizing back in the first mode we have that
$$ Y_I(t_1) =  \widetilde{C}_0 \times_1 \widetilde{\textbf{U}}_0 \bigtimes_{i=2}^d \textbf{U}_0 \ .$$
Taking this as initial condition and iterating the same process for all the successive modes we obtain the updated anti-symmetric tensor 
$$ X_0 = \widetilde{C}_0 \bigtimes_{i=1}^d \widetilde{\textbf{U}}_0 \ .$$
We now apply the anti-symmetric low-rank integrator to the multi-particle interaction,
$$ \dot{X} = -W[X], \qquad X(t_0) = X_0 , $$
where
$$ 
	W[Y] = 	\sum_{l=1}^{d} 
			\sum_{k=l+1}^d 
			Y \times_l \textbf{V}_{\cos} \times_k \textbf{V}_{\cos} + 
			Y \times_l \textbf{V}_{\sin} \times_k \textbf{V}_{\sin} \, .
$$
The core $C_0$ is renormalized after each step, since the absolute size of the tensor is irrelevant.
We emphasize the fact that all along the implementation, it is crucial to use the structure of the Tucker tensor and avoid to build the huge matrix $\textbf{V}_0$ appearing in the definition of the integrator. The \textbf{K} and $C$ problems are linear and can be solved with few iterations of the Arnoldi process. 

In our numerical experiment we choose $d=3$ particles in 1 space dimension and fix the  number of Fourier basis functions per particle at $K=128$, the step-size at $h=0.01$ and we propagate the system until $T=40$.

We introduce the discrete energy
$$ 
E(Y) = \biggl(\frac{2\pi}{K}\biggr)^d \, \langle  Y, H[Y] \rangle_F\,.
 $$
Although the integrator preserves the anti-symmetry in theory, in a straightforward implementation round-off errors will destroy the anti-symmetry and take the system to the lowest state of energy: the bosonic ground state - the one achieved starting from a symmetric initial value. This behavior can be corrected in the integrator by enforcing the anti-symmetry of the small core tensor (which is violated only by round-off errors) at each step or ever after a few steps. In this way the computation tends to the fermionic ground state. 
The energy levels generated by the approximation of rank 5 
are shown in Figure \ref{fig:GroundState} for the bosonic system and the fermionic system with and without enforced anti-symmetrization.

\begin{figure}
	\includegraphics[width=\textwidth]{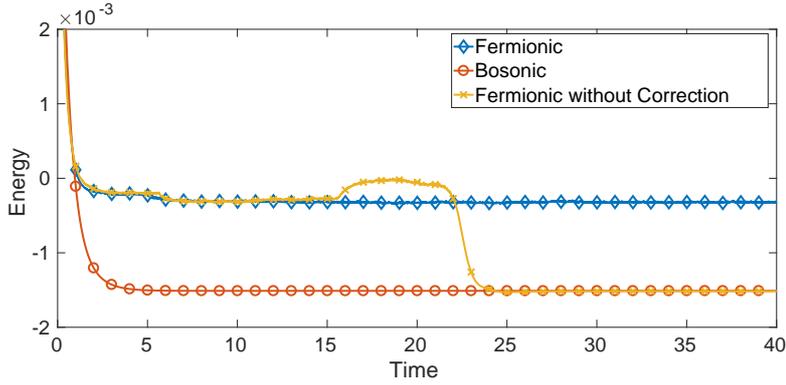}
	\caption{Evolution to the fermionic ground state energy computed with rank 5.}
	\label{fig:GroundState}
\end{figure}

\subsection{MCTDHF - Ultra fast laser dynamics}
In the second example we consider the situation where an external pulsing laser field is introduced in the system. We refer to \cite{Caillat:MCTDHF} for the physical description of the problem and its MCTDHF formulation. We consider the time-dependent Schr{\"o}dinger equation 

\begin{equation*}
	i \partial_t \psi = \mathcal{H}(t) \psi, \quad \psi(t_0) = \psi_0,
\end{equation*}
with the Hamiltonian 
	$$ 
	\mathcal{H}(t) := \sum_{l=1}^{d} \frac{1}{2} 
	\Big[
	\frac{1}{i} \partial_l - \omega(t)
	\Big]^2 +
	V(x_l) -
	\sum_{k=l+1}^{d}
	V(x_l - x_k).
	$$
with the torsion potential $V$ as before and with parameters
	$$
	\omega(t) := A_0 e^{-t^2 / \tau^2} \sin(\Omega t),
	\quad
	A_0 = 100,
	\quad
	\Omega = 100 , 
	\quad
	\tau = 0.2 \pi \ .
	$$	
As initial value, we choose the ground-state calculated at the previous step. 
We fix the  number of Fourier basis functions per particle at $K=128$, the step-size at $h=0.005$ and we propagate the system until $T=1$ by the same algorithm as in the previous subsection, but this time for the real-time evolution instead of the imaginary-time evolution.

The time evolution of the energy obtained by the approximation of rank 5 is shown in Figure~\ref{fig:energy}.

\begin{figure}
	\includegraphics[width=\textwidth]{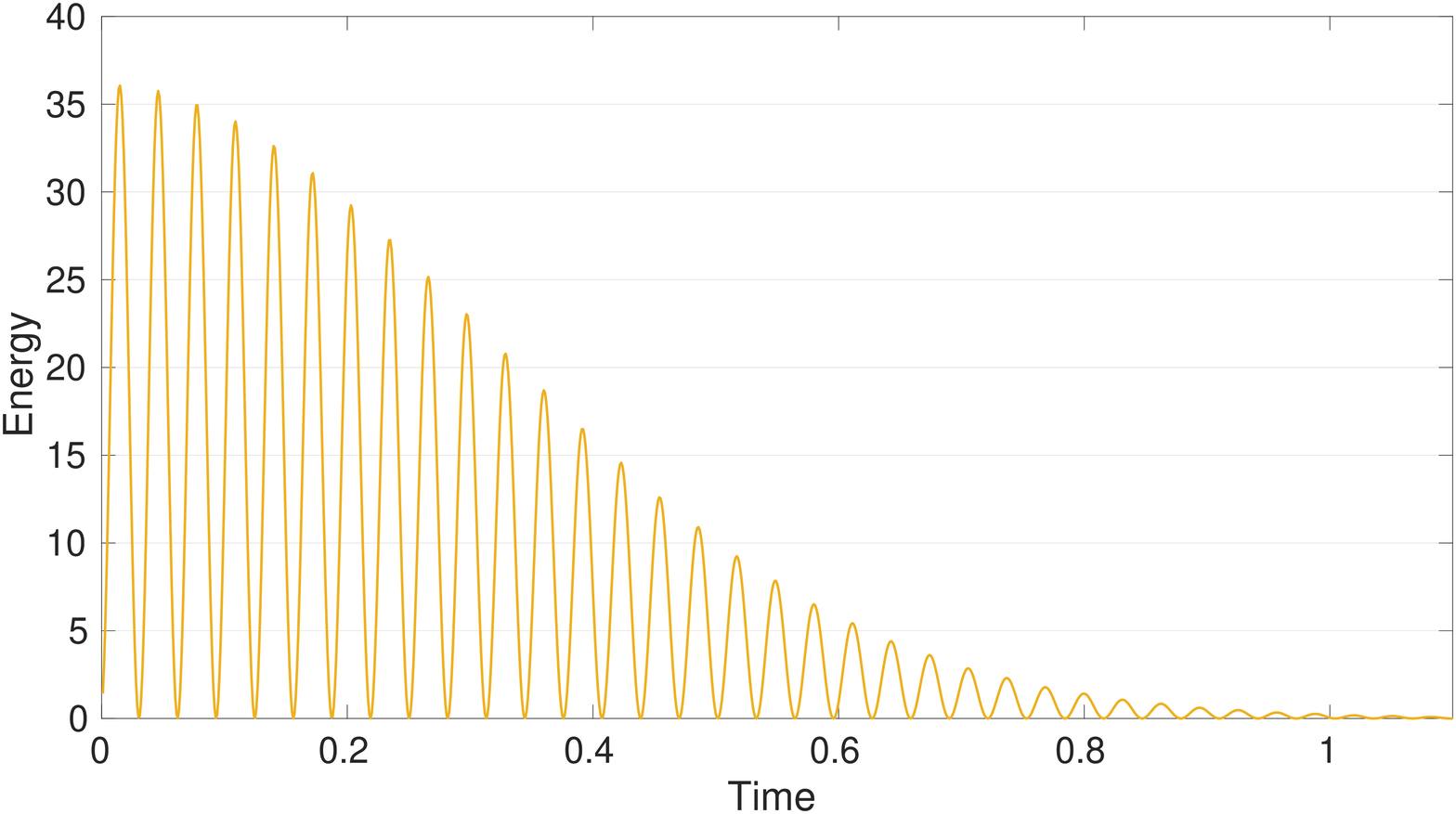}
	\caption{Energy evolution computed with rank 5.}
	\label{fig:energy}
\end{figure}

\begin{acknowledgements}
	 We thank Bal\'azs Kov\'acs and Hanna Walach for their constructive comments and suggestions. This work was supported by Deutsche Forschungsgemeinschaft, Graduiertenkolleg 1838 ``Spectral Theory and Dynamics of Quantum Systems". 
\end{acknowledgements}

\bibliographystyle{plain}
\bibliography{dlrsym}

\end{document}